\documentclass{article}

\usepackage{amsthm,amsmath,amssymb}
\usepackage{graphicx,color}
\usepackage{enumitem}
\usepackage[margin=1in]{geometry}
\usepackage{comment}

\newtheorem{theorem}{Theorem}

\newtheorem{lemma}{Lemma}

\newtheorem{corollary}{Corollary}

\newcommand{\var}{\textrm{Var}}
\newcommand{\Z}{\mathbb{Z}}

\newcommand{\R}{\mathbb{R}}

\newcommand{\ep}{\epsilon}
\newcommand{\sgn}{\mathrm{sgn}}
\renewcommand{\P}{\mathbb{P}}
\newcommand{\E}{\mathbb{E}}
\newcommand{\X}{\mathcal{X}}
\newcommand{\0}{\mathbf{0}}
\newcommand{\poi}{\mathrm{Poisson}}
\newcommand{\unif}{\mathrm{Uniform}}
\newcommand{\F}{\mathcal{F}}
\newcommand{\A}{\mathcal{A}}
\newcommand{\1}{\mathbf{1}}
\author{Eric Foxall}
\title{Stochastic calculus and sample \\path estimation for jump processes}
\begin{document}
\maketitle

\begin{abstract}
We describe stochastic calculus in the context of processes that are driven by an adapted point process of locally finite intensity and are differentiable between jumps. This includes Markov chains as well as non-Markov processes. By analogy with It{\^o} processes we define the drift and diffusivity, which we then use to describe a general sample path estimate. We then give several examples, including ODE approximation, processes with linear drift, first passage times, and an application to the stochastic logistic model.
\end{abstract}

\section{Introduction}\label{sec:intro}
In this article we obtain a general theory of stochastic calculus for processes whose randomness is driven by a compound point process of random and locally finite intensity. Naturally, this includes, but is not limited to, Markov chains, nor does the Markov property does not need to be assumed. The main goal is to demonstrate the simplicity and flexibility of the theory in the context of sample path estimation, and to give a common framework for a growing number of examples in the literature.\\

Existing research on sample path estimation includes the early work of Kurtz (\cite{kurtzDE}, see also \cite{kurtzDE2} and references within) in the context of approximating Markov chains by solutions to ODEs. Later work of Darling and Norris \cite{darlingnorris} contains similar estimates as well as several examples, and a different method of proof. An extension to certain Markov chains on state spaces with countably many coordinates is given in \cite{barzak2}. In each case, sample path estimates are obtained by working with either quadratic or exponential martingales.\\

In the spirit of It{\^o} calculus, we first define drift and diffusivity processes, and obtain rules of differentiation for the drift. Using these rules we derive an exponential local martingale that we use to give a useful sample path estimate in terms of the compensator (indefinite integral of the drift) and predictable quadratic variation (indefinite integral of the diffusivity). We then discuss several ways in which this estimate can be used to control sample paths.\\

The paper is organized as follows. In Section \ref{sec:main} we give the main results concerning existence of local martingales, processes falling within this class, stochastic calculus and sample path estimation. In Section \ref{sec:app} we demonstrate several ways in which the sample path estimate can be used in practice, including ODE approximation, tail estimates for processes with linear drift, first passage time bounds, and an application to the stochastic logistic model. Section \ref{sec:proof} contains proofs of the main results.\\

In order to maintain a fairly lightweight theory, we've chosen to focus on processes with bounded jump size and locally finite jump intensity, and without any continuous martingale (i.e., Brownian) terms.
 In addition, we take a constructive approach, which is again lighter on the theory, and better suited to applications and specific examples.
 However, we expect that with appropriate assumptions, some of our results, such as the general sample path estimate, apply to a larger class of processes.\\

\section{Definition and Main Results}\label{sec:main}
In as general a form as possible, we consider a stochastic process $X = (X_t)_{t \geq 0}$ in continuous time that jumps in response to an underlying point process with finite intensity, and is differentiable between jumps. The goal is to obtain a class of processes that
\begin{enumerate}[noitemsep]
\item is closed under the usual operations on functions such as pointwise\\addition, scaling, multiplication, composition, and integration,
\item has a well-defined notion of drift and diffusivity, and
\item is such that zero drift processes are local martingales.
\end{enumerate}

We record some definitions and notation. We denote the state space $\X$, which we assume is a normed space. For a function $f$ let
$f(a^-) = \lim_{x\to a^-}f(x) \ \hbox{and} \ f(a^+) = \lim_{x \to a^+} f(x).$
Recall that a function $f$ from $\R_+$ into a metric space is \emph{right-continuous with left limits} (rcll) if
$$f(a^-) \ \hbox{exists and} \ f(a^+) =f(a) \ \hbox{for each} \ a.$$
Left-continuous with right limits (lcrl) is defined similarly. For an rcll $f$ define the \emph{jump part} $\Delta f$ by $\Delta f(x) = f(x)-f(x^-)$. Note that $f - \Delta f$ is continuous.
 Let $\zeta$ be a stopping time, and say that a property \emph{holds locally} on $[0,\zeta)$ if there is a \emph{localizing sequence}, that is, an increasing sequence $(\tau_n)$ of stopping times with $\lim_{n\to\infty} \tau_n = \zeta$, such that the property holds on $[0,\tau_n)$ for each $n$.
 Let $(\Omega,\A,\P)$ be a probability space and $\F = (\F_t)$ a filtration. A process $X:\R_+\times\Omega \to \X$ is \emph{progressively measurable} if $(t,\omega)\mapsto X_t(\omega)$ is measurable and $X$ is adapted to $\F$. The same definition applies if $X$ is defined only on $\{(t,\omega):t < \zeta(\omega)\}$.\\

We now describe the probability space. Let $\{(w_i,u_i):i=1,2,\dots\}$ be an i.i.d. family of random variables, each $w_i$ exponentially distributed with mean $1$, and each $u_i$ uniformly distributed on $[0,1]$ and independent of $w_i$. For $i\ge 1$ let $t_i = \sum_{j=1}^i w_i$, so that $(t_i)_{i\ge 1}$ is the ordered set of points in a Poisson point process with intensity 1. Let $(\Omega,\A,\P)$ denote the corresponding probability space.\\

Next we describe the transition rate. Let $q$ be a stochastic process on $(\Omega,\A,\P)$ with values in $\R_+$, such that $(t,\omega) \mapsto q_t(\omega)$ is measurable. Define $r$ by $r_t = \sup\{r:\int_0^rq_sds < t\}$ for $t\ge 0$, and for $i\ge 1$ let $r_i = r_{t_i}$, so that
$$J = \{r_i:i\ge 1\}\quad\hbox{and}\quad K = \{(r_i,u_i):i\ge 1\}$$
are respectively the set of jump times, and the jump times with sample values at each jump, of a Poisson point process with time-dependent intensity $q_t$. For a Borel set $B \subset \R_+$ let $JB = J\cap B$ and $KB = \{(s,u) \in K:s\in B\}$, and let $J_t = J [0,t]$ and $K_t = K[0,t]$. Let $\F = (\F_t)$ denote the natural filtration for $(K_t)$.\\

To account for explosion, we restrict to the time interval $[0,\zeta)$ where $\zeta=\lim_{n\to\infty}\inf\{t:q_t \ge n\}$, which ensures the jump rate is locally finite. This is a natural assumption, for example, for Markov chains, since in that case $\zeta$, when finite, corresponds to the explosion time. For later use let $[0,\zeta) \times \Omega$ denote the set $\{(t,\omega) \in \R_+ \times \Omega: t< \zeta(\omega)\}$. We now describe the process, taking a constructive approach. Our generic process $X$ takes values in a normed space $(\X,|\cdot|)$, and given $X_0 \in \X$, is defined for $t \in [0,\zeta)$ by
\begin{equation}\label{eq:constr-proc}
X_t = X_0 + \int_0^t D_sds + \sum_{(s,u) \in K_t}\Delta_s(u),
\end{equation}
using the data
$$\begin{array}{rcll}
q &:& [0,\zeta)\times\Omega\to\R_+ & \hbox{the transition rate,}\\
D &:& [0,\zeta)\times\Omega\to \X & \hbox{the derivative, and}\\
\Delta &:& [0,\zeta)\times\Omega \to L^{\infty}([0,1],\X) & \hbox{the jump function}.
\end{array}$$
We make the following assumptions on $(q,D,\Delta)$.
\begin{enumerate}[noitemsep]
\item $(q,D,\Delta)$ is progressively measurable with respect to $\F$.
\item $t\mapsto \Delta_t$ is left-continuous.
\item $\int_0^t|D_s|ds$ is locally finite, that is, $(\zeta\wedge\inf\{t:\int_0^t |D_s|ds = n\})_{n\ge 0}$
is a localizing sequence for $\zeta$.
\item Bounded jumps: for some $c_\Delta<\infty$ and a.e. $\{(t,u,\omega): t < \zeta(\omega)\}$, $|\Delta_t(u,\omega)| \le c_\Delta$.
\end{enumerate}
For lack of a better term, we refer to such a process $X$ as a \emph{hybrid jump process} or hjp for short, since it has in general both a jump component and can vary between jumps. Based on our assumptions, $X$ is progressively measurable with respect to $\F$, is rcll, and is absolutely continuous at all points $t \notin J$. It has the decomposition
$$X_t - \Delta X_t = X_0 + \int_0^t D_sds \quad \hbox{and} \quad \Delta X_t = \Delta_t(u)\1((t,u) \in K),$$
and a.s., $|\Delta X_t| \le c_\Delta$ for all $t\ge 0$. As a shorthand, given an hjp $X$ we use $q(X),D(X),\Delta(X)$ to denote the data, with the subscript like $q_t(X)$. Note the difference between $\Delta X_t$, the jump in $X$ at time $t$, and $\Delta_t(X)$, the jump function of $X$ at time $t$.\\

We first characterize the set of hjp that are local martingales, and identify the analogue of Dynkin's martingale. Note this does not follow from the corresponding result for Feller processes, as we have not assumed even that $X$ is Markov.

\begin{theorem}[Compensation]\label{thm:compensate}
Let $X$ be a hybrid jump process with data $q,D,\Delta$, and define the \emph{drift} $\mu(X)$ by
$$\mu_t(X) = D_t + q_t \int_0^1 \Delta_t(u)du.$$
Then $X$ is a local martingale (submartingale) if $\mu_t(X) = 0$ ($\mu_t(X) \ge 0$) for a.e. $(t,\omega) \in [0,\zeta)\times \Omega$. In particular, if we define the \emph{compensator} $\bar{X}$ by
$$\bar{X}_t = X_0 + \int_0^t \mu_s(X)ds$$
then for any hjp $X$ the \emph{compensated process} $M(X)$ defined by
$$M_t(X) = X_t - \bar{X}_t$$
is a local martingale.
\end{theorem}

Next we show that hybrid jump processes includes not only Markov chains but also mixtures of Markov chain and solutions to ordinary differential equations. We state the more general result first, then show how it includes Markov chains.

\begin{theorem}[Hybrid markov process]\label{thm:hmp}
Let $\X$ be a normed space and let
$$q:\X\to\R_+, \quad
D:\X\to\X, \quad \hbox{and} \quad
\Delta:\X \to L^{\infty}([0,1],\X)$$
be such that
\begin{enumerate}[noitemsep]
\item $x \mapsto D(x)$ is locally Lipschitz, and
\item for some $c_\Delta$, $\sup_{x \in \X}\|\Delta(x)\|_{L^{\infty}} \le c_\Delta$.
\end{enumerate}
Then given $X_0$, there is a unique hjp $X$ on time interval $[0,\zeta)$ where
$$\zeta = \lim_{n\to\infty}\inf\{t:q(X_t) \ge n \ \hbox{or} \ \int_0^t |D(X_s)|ds \ge n \},$$
satisfying \eqref{eq:constr-proc} with $q_t(X) = q(X_t), \ D_t(X) = D(X_t)$ and $\Delta_t(X) = \Delta(X_{t^-})$.
 If, in addition, $q$ is a locally bounded function on $\X$, that is,
 $\sup_{x \in \X:|x| \le r} q(x)<\infty$ for each $r >0$, then we can take for $\zeta$ the first escape time of $X_t$, that is,
$$\lim_{n\to\infty}\inf\{t:|X_t|\ge n\}.$$
\end{theorem}

To write a Markov chain in this framework, proceed as follows. Suppose $S$ is a countable subset of a normed space $\X$ and rates are given by a transition rate matrix $\{q_{ij}:i,j \in S\}$. Define the collection of functions $\{\Delta_i,q_i:i\geq 1\}$ with $\Delta_i:S\to\X$ and $q_i:S\to\R_+$ by taking $q_j(i) = q_{ij}$ and $\Delta_j(i) = j-i$ for $i\ne j$, and $q_j(j)=\Delta_j(j)=0$. For $x\in \X$ let $q(x) = \sum_i q_i(x)$. Let $c_i(x) = q(x)^{-1}\sum_{j=1}^i q_j(x)$ for $i\geq 1$, then define $\Delta(X)$ by $\Delta_t(X)(u) = \Delta(X_t,u)$, where
$$\Delta(x,u) = \sum_i \Delta_i(x)\mathbf{1}(c_{i-1}(x)< u \leq c_i(x)).$$
Letting $q_t(X) = q(X_t), D(X)\equiv 0$ and $\Delta_t(X) = \Delta(X_t)$ defines the process. Then, the Poisson thinning property shows it has the correct transition rates.\\

Next we describe the closure properties of hjp, that is, in what sense the class of hjp is closed under the usual operations of addition, multiplication, indefinite integral etc., and describe how the drift behaves under these operations. The fact that even routine operations on Markov processes lead to non-Markov processes (assuming the state space is left unchanged) was an important motivator for the definition of hjp.\\

In what follows, note that $\smash{\E[ \Delta_t(X) \mid \F]=\int_0^1 \Delta_t(X)(u)du}$ -- the first expression is used when we think it conveys the meaning more clearly. In addition, given $X$ we define the left-continuous process $X^-$ by $X_t^- = X_{t^-}$ for $t\ge 0$.

\begin{theorem}[Stochastic calculus]\label{thm:stocalc}
Fix a transition rate $q$ and a normed algebra $(\X,|\cdot|)$, and let $\F$ denote the filtration induced by $q$. Let $X,Y$ be hybrid jump processes (hjp) with common transition rate $q$ and state space $\X$. Let $f:\X\to\X$ and $g:\R_+\to\X$ be absolutely continuous functions; note that $g$ can be viewed as an $\F$-adapted hybrid jump process with data $q,g',0$. Let $\tau$ be an $\F$-stopping time. Define the \emph{covariability} $\sigma(X,Y)$ and \emph{predictable covariation} $\langle X,Y \rangle$ by
$$\sigma_t(X,Y) = q_t \E[ \ \Delta_t(X)\Delta_t(Y) \mid \F \ ] \quad\hbox{and} \quad \langle X,Y \rangle_t = \int_0^t \sigma_s(X,Y)ds,$$
and define $\Delta(f(X))$ and $I(X,Y)$ by
$$\begin{array}{rcl}
\Delta_t(f(X)) &=& f(X_t + \Delta_t(X)) - f(X_t) \quad \hbox{and}\\
I_t(X,Y) &=& \int_0^t\mu_s(X)Y_sds.
\end{array}$$
Then, the following are hjp with transition rate $q$ and $D,\Delta$ as shown. Below, $a \in \R$ is a constant. Also, for $XY$ require that $Y$ is bounded, i.e., for some $C>0$, a.s. $\sup_t |Y_t| \le C$, and for $f(X)$ require that $X$ is bounded, or $f$ is Lipschitz.\\
\begin{center}
\begin{tabular}{l l l l}
process & derivative & jump function & drift \\
\hline \vspace*{-5pt} \\
$g$ & $g'$ & $0$ & $g'$\\
$X^{\tau}$ & $D(X)^{\tau}$ & $\Delta(X)^{\tau}$ & $\mu(X)^\tau$ \\
$X + aY$ & $D(X)+aD(Y)$ & $\Delta(X) + a\Delta(Y)$ & $\mu(X) + a\mu(Y)$\\
$XY$ & $D(X)Y + XD(Y)$ & $\Delta(X)Y^- + X^-\Delta(Y) + \Delta(X)\Delta(Y)$ & $\mu(X)Y + X\mu(Y) + \sigma(X,Y)$\\
$f(X)$ & $f'(X)D(X)$ & $\Delta(f(X))$ & $f'(X)D(X) + \Delta(f(X))$\\
$I(X,Y)$ & $\mu(X)Y$ & $0$ & $\mu(X)Y$ \\
$\langle X,Y \rangle$ & $\sigma(X,Y)$ & $0$ & $\sigma(X,Y)$ \\
\end{tabular}
\end{center}
In particular, we note the following rules.
\begin{enumerate}[noitemsep]
\item \textbf{Deterministic function.} $\mu(g)=g'$.
\item \textbf{Linearity.} $\mu(X+aY) = \mu(X)+a\mu(Y)$, for $a \in \R$. 
\item \textbf{Product rule.} $\mu(X Y) = \mu(X)Y + X\mu(Y) + \sigma(X,Y)$.
\item \textbf{Chain rule.} $\mu(f(X)) = f'(X)D(X) + q\E[\Delta(f(X)) \mid \F]$.
\item \textbf{Indefinite integral.} for $I(X)$ given by $I_t(X) = \int_0^tX_sds$, $\mu(I(X)) = X$.
\end{enumerate}
Also, we have \textbf{Taylor approximation.} Suppose $\X=\R$ and $f \in C^2(\R)$. Define the \emph{diffusivity}
$$\sigma_t^2(X) = \sigma_t(X,X) = q_t\E[ \ \Delta_t^2(X) \mid \F \ ].$$
Then for $t\ge 0$,
$$|\mu_t(f(X)) - f'(X_t)\mu_t(X)| \le \frac{1}{2}\sigma_t^2(X)\sup_{|x-X_t| \le c_\Delta}|f''(x)|.$$
\end{theorem}

Using our theory, given an hjp we obtain a family of exponential local martingales, to which an application of Doob's maximal inequality, combined with the Taylor approximation of Theorem \ref{thm:stocalc}, yields a two-parameter family of sample path estimates. To state it we first define the \emph{predictable quadratic variation} $\langle X\rangle$ of $X$ by
$$\langle X \rangle_t = \int_0^t \sigma^2_s(X)ds.$$

\begin{theorem}[General sample path estimate]\label{thm:sampath}
Suppose $X$ is an hjp with $\X=\R$ defined on time interval $[0,\zeta)$, with jump size bounded by $c_\Delta>0$. Then, for $\lambda,a>0$ and $\bullet \in \pm$,
\begin{equation}\label{eq:sampath}
\P( \bullet M_t(X) \ge a + \frac{\lambda}{2}e^{\lambda c_\Delta} \langle X \rangle_t \ \hbox{for some} \ t < \zeta ) \leq e^{-\lambda a}.
\end{equation}
\end{theorem}

In particular, with probability at least $1-2e^{-\lambda a}$, $|M_t(X)|$ never exceeds $a + \frac{\lambda}{2}e^{\lambda c_\Delta}\langle X \rangle_t$. Recall that $M(X) = X-\bar X$ and let $W_t = a + \frac{\lambda}{2}e^{\lambda c_\Delta}\langle X \rangle_t$ denote the right-hand side in the above event. We can then interpret the result as follows. The compensator $\bar X$ gives us in some sense our best guess of $X$ by a left-continuous process. Then, the envelope $\bar X \pm W$ gives us a gauge of how far away to expect $X$ to be from $\bar X$. What Theorem \ref{thm:sampath} tells us that we have a good chance (at least $1-2e^{-\lambda a}$) of finding $X$ within the envelope $\bar X \pm W$ for all time.

\section{Applications}\label{sec:app}
In order to simplify certain calculations, we begin with a reformulation of Theorem \ref{thm:sampath}.
 For $c>0$ the function $\lambda \mapsto \gamma_c(\lambda) = \lambda e^{\lambda c}/2$ is increasing and tends to $\infty$, so has a functional inverse $\lambda_c(\gamma)$ which is defined for $\gamma \in (0,\infty)$ and is also increasing. Let $\Gamma$ denote the function $x \mapsto xe^x/2$ and $\Gamma^{-1}$ the inverse function, and define $\psi$ by $\psi(y) = \Gamma^{-1}(y)/y$. Notice that $\gamma_c(\lambda) = \Gamma(\lambda c)/c$, so inverting gives $\lambda_c(\gamma) = \Gamma^{-1}(c\gamma)/c = \psi(c\gamma)\gamma$. If we let
$$\kappa_c(\gamma,a) = e^{\lambda_c(\gamma)a} = e^{\Gamma^{-1}(c\gamma)a/c} = e^{\psi(c\gamma)\gamma a},$$
then we can restate \eqref{eq:sampath} by saying that for $\lambda,a>0$ and $\bullet \in\pm$,
\begin{equation}\label{eq:ld}
\P(\bullet M_t(X) \ge a + \gamma \langle X \rangle_t  \ \hbox{for some} \ t < \zeta) \leq 1/\kappa_{c_\Delta(X)}(\gamma,a).
\end{equation}
Notice that $\Gamma$ is convex with fixed points $0$ and $\log 2$, and $\Gamma'(0)=1/2$. So, $\Gamma^{-1}$ is concave with the same fixed points and $(\Gamma^{-1})'(0)=2$, which means that $\psi$ is decreasing, takes values in $(0,2)$, $\psi(y) \ge 1$ for $y \le \log 2$ and $\lim_{y\to 0^+}\psi(y)=2$. Using the last expression for $\kappa$, which is perhaps the most helpful, $\kappa_{c_\Delta}(\gamma,a) \ge e^{\gamma a}$ if $c_\Delta\gamma \le \log 2$, and $\kappa_{c_\Delta}(\gamma,a) \sim e^{2\gamma a}$ as $c_\Delta\gamma \to 0$. In particular, $\kappa_{c_\Delta}(\gamma,a) \to e^{2\gamma a}$ as $c_\Delta \to 0$, when $\gamma,a$ are kept fixed.\\

In practice, it is often enough to estimate $\sigma_t^2(X)$ by the transition rate and jump size as follows. Letting $c_{\Delta,t}(X) = \|\Delta_t(X)\|_{L^{\infty}}$, $\sigma^2_t(X) \le \rho_t(X) := q_t c_{\Delta,t}^2(X)$, which can be plugged into \eqref{eq:ld} to give
$$\P(\bullet M_t(X) \ge a + \gamma \int_0^t \rho_s(X)ds \ \hbox{for some} \ t < \zeta) \leq 1/\kappa_{c_\Delta(X)}(\gamma,a).$$
Then, if $c_\rho(X) = \sup_{t,\omega} \rho_t(X(\omega))$ is finite, the error term is at most $a + \gamma c_\rho(X)t$. If we want a bound on a fixed time horizon, we can then take $a + \gamma c_\rho(X)T$ as the error and optimizing $\gamma a$ subject to $a+\gamma c_\rho(X)T = \delta$, obtain $a=\delta/2$ and $\gamma = \delta/(2c_\rho(X)T)$, so $\gamma a = \delta^2/(4c_\rho(X)T)$ and the estimate
$$\P(\bullet M_t(X)\ge \delta \ \hbox{for some} \ t < T \wedge \zeta) \le \exp(-\psi(c_\Delta(X)\delta/(2c_\rho(X)T))\delta^2/(4c_\rho(X)T)).$$
\noindent\textbf{ODE approximation.} If $X$, in addition, has drift $\mu_t(X) = \mu(X_t)$ for some Lipschitz function $\mu:\X\to\X$, then using Gronwall's inequality as described in \cite{kurtzDE}, we obtain the estimate
\begin{equation}\label{eq:ODEest}
\P( \sup_{t \le T\wedge\zeta}|X_t - \phi_t(X_0)| \ge e^{LT}\delta) \le \exp(-\psi(c_\Delta(X)\delta/(2c_\rho(X)T))\delta^2/(4c_\rho(X)T)),
\end{equation}
where $(t,x) \mapsto \phi_t(x)$ is the flow corresponding to the ODE $y' = \mu(y)$. Immediately this gives a strong estimate for some sequences of processes. Recall from \cite{kurtzDE} the definition of a density dependent Markov chain $X$ with $\X=\Z^d$ and transitions $q_{k,k+\ell} = nq(n^{-1}X_t,\ell)$ for some function $q:\R^{2d}\to \R_+$ and parameter $n$. Here we assume also that $\sum_{\ell}q(k,\ell)\le c_q$ for each $k \in \Z^d$ and $q(k,\ell)=0$ if $|\ell-k| \ge r$, for some $c_q,r>0$. Then, for the rescaled process $x:=n^{-1}X$, $q_t(x) \le nc_q$ and $c_\Delta(x) = n^{-1}r$, so $\rho_t(x) \le n^{-1}c_qr$ which we can take to be $c_\rho$. The argument to $\psi$ in the right-hand side of \eqref{eq:ODEest} becomes $\delta/2c_qT$, and $\psi(\delta/2c_qT)\to 2$ if $\delta \to 0$ while $T$ is kept fixed. In fact, taking $\delta = f(n)n^{1/2}$ with $f(n)=o(n^{1/2})$, the right-hand side of \eqref{eq:ODEest} $\sim \exp(-2 f(n)^2/4c_qrT)$ as $n\to\infty$. A similar upper bound on the probability holds if $f(n)=O(n^{1/2})$, just with a smaller constant than 2. This type of result is not new, but is included to demonstrate the ease with which explicit probability estimates can be obtained.\\

\noindent\textbf{Linear drift.} The next result controls the growth of a non-decreasing hjp with bounded jumps and linear (or more generally, sublinear) drift. It shows that the largest value ever reached by the normalized process $Y$ has an exponential tail. To obtain this result we rely on the fact that $\langle Y \rangle_t \le y c_\Delta(X)$ so long as $\sup_{s \le t}Y<y$.

\begin{lemma}[Linear drift]\label{lem:lindrift}
Let $(X_t)_{t\ge 0}$ be a non-decreasing hjp on $\X=\R_+$ such that
\begin{equation}\label{eq:lindrift}
\mu_t(X) \le \ell(t)X_t
\end{equation}
for some locally integrable deterministic function $\ell(t)$. Let $m(t) = \exp(\int_0^t \ell(s)ds)$ and let $Y_t = X_t/(X_0 m(t))$ denote the rescaled process. Let $\zeta' = \zeta \wedge \inf\{t:m(t)=\infty\}$. Then, for $y\geq 2$,
$$\P(\sup_{t < \zeta'} Y_t \geq y) \leq \E[e^{-(y-2)X_0/4c_\Delta(X)}]$$
If the non-decreasing assumption is replaced with the assumption $\sigma^2(X) \leq C\mu(X)$ a.s. for some $C>0$, then the same estimate holds with $\max(C,c_\Delta(X))$ in place of $c_\Delta(X)$.
\end{lemma}

\begin{proof}
First we treat the case $X_0=1$, so that $Y_t = X_t/m(t)$. Let $c=c_\Delta(X)$ and $q_t=q_t(X)$. Given $y>0$ define $\tau(y) = \inf\{t:Y_t \geq y\}$.
 Since $1/m(t) = e^{-\int_0^t\ell(s)ds}$, $(1/m(t))' = -\ell(t)/m(t)$, so using the product rule on $Y_t = X_t/m(t)$,
$$\mu(Y_t) \le \ell(t)X_t/m(t) + X_t(-\ell(t)/m(t)) = 0 \quad\hbox{and}\quad \bar Y \le Y_0.$$
Clearly $\sigma^2_t(Y) = (1/m(t))^2\sigma^2_t(X)$. 
 Since $X$ is non-decreasing, $0 \le \Delta_t(X) \le c$ and
$$\sigma^2(X_t) = q_t \int_0^1\Delta_t^2(X)(u)du \leq q_tc\int_0^1\Delta_t(X)(u)du.$$
Since $X$ is non-decreasing, $X_t' \geq 0$, so
$$q_t\int_0^1 \Delta_t(X)(u) \leq X_t' + q_t\int_0^1 \Delta_t(X)(u)du \leq \mu_t(X).$$
Combining, $\sigma^2_t(X) \leq c\mu_t(X)$. Using $\mu_t(X)=\ell(t)X_t = \ell(t)m(t)Y_t$,
$$\sigma^2_t(Y) \leq (1/m(t))^2c\mu_t(X) = (c/m(t))\ell(t) Y_t.$$
Since $Y_t <y$ for $t<\tau(y)$, $\langle Y \rangle_{\tau(y)} \leq y c\int_0^{\tau(y)}\ell(s)/m(s)ds = yc \alpha(\tau(y))$. Taking the antiderivative,
$$\alpha(t) = \int_0^t e^{-\int_0^s \ell(r)dr}\ell(s)ds = 1-e^{-\int_0^t \ell(s)ds} = 1-1/m(t) \leq 1 \quad\hbox{for all} \quad t \geq 0.$$
Since $Y_0=1$ and $Y_{\tau(y)} \geq y$, it follows that
$$M_{\tau(y)}(Y)-\gamma \langle Y \rangle_{\tau(y)} \geq y - 1 - yc \gamma.$$
Using \eqref{eq:ld} with $a=y-1-yc\gamma$,
$$\P(\sup_{t<\zeta'} Y_t \geq y) \leq 1/\kappa_c(\gamma,a).$$
Optimizing $\gamma a$ gives $\gamma = (y-1)/2yc$. If $y \ge 1$ then $c\gamma \le 1/2 \le \log 2$ and $\psi(c\gamma) \ge 1$ and
$$\kappa_c(\gamma,a) \ge (y-1)^2/4yc \geq (y - 2)/4c.$$
If instead we assume $\sigma^2_t(X) \leq C\mu_t(X)$, the same reasoning gives again $M_{\tau(y)}(Y)-\langle Y\rangle_{\tau(y)} \ge y-1 - yC\gamma$. Taking $\gamma = (y-1)/2y\max(C,c)$ which is at most $1/2c$, this is at least $(y-1)/2$ and
$$\gamma(y-1-yc\gamma) \ge (y-1)^2/4y\max(C,c) \geq (y-2)/4\max(C,c).$$
To treat general $X_0$, first condition on $X_0$ and apply the above to $X_t/X_0$, which has jump size $c/X_0$. Then, integrate over $X_0$ to obtain the result.
\end{proof}

\noindent\textbf{First passage times.} Next we derive some general first passage estimates for hjp with $\X=\R$, as a function of the drift, diffusivity and jump size. When the transition rate is bounded, we obtain scaling limits as $c_\Delta(X)\to 0$. For the next three lemmas we let $T_x = \inf\{t:X_t \geq x\}$ for $x>0$. Note that in the examples so far we have treated $M(X)$ and $a + \gamma \langle X\rangle_t$ as being fairly separate. However, if we ``unwrap'' the inequality $\pm M_t(X) - \gamma \langle X \rangle_t \ge a$ and view it as
$$\pm(X_t - X_0) - \int_0^t ( \pm\mu_s(X)ds + \gamma\sigma_s^2(X))ds \ge a,$$
then we can obtain estimates of $X_t-X_0$, which is the approach we take below. 

\begin{lemma}[Drift barrier]\label{lem:driftbar}
 Let $X$ be an hjp with $\X=\R$ and suppose $x>c=c_\Delta(X)$. Suppose there are $\mu,C_\mu,\sigma^2>0$ so that
 $$0<X_t<x \quad \hbox{implies} \quad \mu_t(X) \leq -\mu,\quad |\mu_t(X)| \leq C_\mu \quad\hbox{and}\quad \sigma^2_t(X) \leq \sigma^2.$$
Let $t_0 = (x-c)/20C_\mu$, $\gamma = \mu/\sigma^2$ and $a = (x-c)/2$, and let $\kappa = \kappa_c(\gamma,a)$. Then, for integer $k\geq 1$,
 $$\P( \ \sup_{t < kt_0 \wedge \zeta}X_t \ge x \ \mid \ X_0 \leq x/2 \ ) \leq 3k/\kappa.$$
 In particular,
 \begin{equation}\label{eq:driftbar}
 \P( \ \sup_{t < \lfloor \kappa^{1/2} \rfloor t_0 \wedge \zeta}X_t \geq x \ \mid \ X_0 \leq x/2 \ ) \leq 3/\kappa^{1/2}.
 \end{equation}
\end{lemma}

\begin{proof}
Suppose $|X_0 - x/2| \leq c/2$ and let $\tau = \inf\{t:|X_t - x/2| \geq x/2\}$. For $t<\tau$ we find
  $$M_t(X) - \gamma \langle X\rangle_t \geq X_t - X_0 + (\mu - \gamma\sigma^2)t$$
  so taking $\gamma = \mu/\sigma^2$ and $a=(x-c)/2$, if $X_{\tau} \geq x$ then $M_\tau(X)-\gamma \langle X \rangle_\tau \ge a$. Using \eqref{eq:ld},
  $$\P(X_{\tau} \geq x \ \mid \ |X_0 - x/2| \leq c/2) \leq 1/\kappa.$$
  On the other hand, since $|\mu_t(X)| \leq C_\mu$ for $t<\tau$ and $|X_{\tau} - X_0| \geq (x-c)/2$,
  $$\max_{\bullet \in \pm}\bullet M_{\tau}(X) - \gamma \langle X\rangle_{\tau} \geq (x-c)/2 - (C_\mu+\gamma\sigma^2)t.$$
  This time take $\gamma$ so that $\lambda_c(\gamma) = 2\lambda_c(\mu/\sigma^2)$ and $a=(x-c)/4$, which gives the same value of $\lambda_c(\gamma)a$ and thus of $\kappa_c(\gamma,a)$ as before, and the lower bound $(x-c)/2 - (C_\mu+2\lambda_c(\mu/\sigma^2)\sigma^2)t$. Taking $t_0 = (x-c)/(4(C_\mu+2\lambda_c(\mu/\sigma^2)\sigma^2))$ gives the lower bound $(x-c)/4$ on $\max_{\bullet \in \pm}\bullet M(X) - \langle X \rangle$, then using both sides of the estimate and taking a union bound,
  $$\P(\tau \leq t_0) \leq 2/\kappa.$$
  Taking a union bound with the previous estimate,
  $$\P(\tau \leq t_0 \quad\hbox{or}\quad X_{\tau} \geq x) \leq 3/\kappa.$$
  Since $\gamma_c(\lambda) = \lambda e^{\lambda c}/2 \geq \lambda/2$, $\lambda_c(\gamma) \le 2\gamma$, so $\lambda_c(\mu/\sigma^2)\sigma^2 \leq 2\mu$. By definition, $\mu \leq C_\mu$, so $t_0 \geq (w-c)/20C_\mu$. Take the latter to be the value of $t_0$. Since it is smaller, the above estimate remains valid. Then, it suffices to iterate the estimate, alternately stopping the process when $|X_t - x/2| \leq c/2$ and $|X_t- x/2| \geq x/2$.
\end{proof}


\begin{lemma}[Drift escape]\label{lem:driftesc}
Let $X$ be an hjp with $\X=\R$ and suppose there are $\mu,\sigma^2>0$ such that
$$\mu_t(X) \ge \mu \quad\hbox{and} \quad \sigma^2_t(X) \leq \sigma^2 \quad \hbox{for} \quad t\le T_x.$$
For $b,\ep>0$ and $\ep<1$, let $\gamma = \ep\mu/\sigma^2$ and $a = bx$, and let $\kappa = \kappa_{c_\Delta(X)}(\gamma,a)$. Let $T=(1+b)x/(1-\ep)\mu$. Then,
$$\P(\sup_{t \le T}X_t < x \mid X_0 \ge 0) \le 1/\kappa.$$
\end{lemma}

\begin{proof}
If $X_0 \ge 0$ and $T_x>T$ then $X_0-X_T \ge -x$ and
$$-M_T(X) - \gamma \langle X \rangle_T \ge -x + (\mu-\gamma\sigma^2)T.$$
Taking $\gamma = \ep\mu/\sigma^2$ and $T=(1+b)x/(1-\ep)\mu$ gives the lower bound $bx$. Using \eqref{eq:ld} then gives the result.
\end{proof}

\begin{lemma}[Diffusive barrier]\label{lem:neutbar}
 Let $X$ be an hjp with $\X=\R$ and suppose that $\mu_t(X) \leq 0$ for $t<T_x$.
 For $T>0$, let $\gamma = x/2\langle X \rangle_T$ and $a = x/2$, and let $\kappa=\kappa_{c_\Delta(X)}(\gamma,a)$. Then,
 $$\P(\sup_{t \leq T}X_t \ge x \mid X_0 \le 0) \leq 1/\kappa.$$
\end{lemma}

\begin{proof}
If $X_t \ge x$ for some $t\leq T$ while $X_0 \le 0$ then $T_x \le T$ and $M_{T_x}(X)-\gamma \langle X \rangle_{T_x} \ge x - \gamma \langle X \rangle_T$ (note $t \mapsto \langle X \rangle_t$ is non-decreasing). Let $\gamma = x/2\langle X \rangle_T)$ to obtain the lower bound $x/2$. Then use \eqref{eq:ld}.
\end{proof}

\begin{lemma}[Diffusive escape]\label{lem:neutesc}
Let $X$ be an hjp with $\X=\R$ and let $T_x = \inf\{t:|X_t| \ge x\}$. Suppose that $|\mu_t(X)| \le C_\mu$, $\sigma^2_t(X) \geq \sigma^2 \ge 4xC_\mu>0$ and $\rho_t(X) \le \rho$ for $t<T_x$. Let $\gamma = (\sigma^2/4)(1/(2x+c_\Delta)^2\rho)$ and $a = bx^2$, and let $\kappa=\kappa_{c_\Delta(X)}(\gamma,a)$. Then for $b>0$,
$$\P(\sup_{t \le 4(b+1)(x/\sigma)^2}|X_t| < x \mid X_0 \ge 0) \le 1/\kappa.$$ 
\end{lemma}
\begin{proof}
For $t<T_x$,
$$\begin{array}{rcl}
\mu(X_t^2) &=& 2X_t \mu_t(X) + \sigma^2_t(X) \ge \sigma^2 - 2xC_\mu  \ge \sigma^2/2 \ \hbox{and} \\
\Delta_t(X^2) &=& (X_t+\Delta_t(X))^2 - X_t^2 = 2X_t\Delta_t(X) + \Delta^2_t(X) \le (2x+c_\Delta)c_{\Delta,t},
\end{array}$$
and so
$$\sigma^2_t(X^2) = q_t(X) \int_0^1 \Delta_t^2(X^2)(u)du \le (2x+c_\Delta)^2q_t(X)c_{\Delta,t}^2 = (2x+c_\Delta)^2\rho_t(X).$$
If $t<T_x$ then $X_t^2 < x^2$ so using $X_0^2 \ge 0$ and the above,
$$-M_t(X) - \gamma \langle X\rangle_t \geq -x^2 + (\sigma^2/2 - \gamma(2x + c_\Delta)^2\rho )t.$$
Take $\gamma = (\sigma^2/4)(1/(2x+c_\Delta)^2\rho)$ to get the lower bound $-x^2 + \sigma^2 t/4$, then let $t = 4(1+b)x^2/\sigma$ to make this at least $bx^2$. Then use \eqref{eq:ld}.\\
\end{proof}

\noindent\textbf{Scaling limits of first passage times.} Using the various forms of $\kappa$ and the following properties of $\Gamma,\Gamma^{-1}$ and $\psi$, we can probe the above estimates in various ways. We recall some properties of $\kappa$.
\begin{enumerate}[noitemsep,label={(\roman*})]
\item If $\gamma,a$ are fixed and $c_\Delta \to 0$ then $c_\Delta\gamma \to 0$ and $\log\kappa \to 2\gamma a$.
\item If $c_\Delta\gamma \le M$ then since $\psi$ is decreasing, $\log\kappa \ge \psi(M)\gamma a$.
\item If $c_\Delta \gamma \ge \delta>0$ then since $\Gamma^{-1}$ is increasing, $\log \kappa \ge \Gamma^{-1}(\delta)a/c_\Delta $.
\end{enumerate}
These become more tangible once we assume the transition rate is bounded, that is, a.s. $\sup_t q_t \le c_q$. We focus on the parameter region $c_qc_\Delta^{\alpha} \le C$ for some $C>0$ and $\alpha \in [1,2]$, with $\alpha=1$ the large deviations regime and $\alpha=2$ the diffusive regime, in the limit as $c_\Delta\to 0$. Estimates break down above $\alpha=2$ in the first three results, while in the last one, they break down for $\alpha<2$. Below, $x>0$ is fixed.
\begin{enumerate}
\item \textit{Drift barrier.} Here, $\gamma=\mu/\sigma^2$ and $a=(x-c_\Delta)/2$, and $t_0$ scales like $x/C_\mu$. If we let $c_\Delta\to 0$ while keeping $C_\mu$ and $\mu/\sigma^2$ fixed, $\log \kappa \to x\mu/\sigma^2$. If $c_q$ bounds the transition rate, then we can take $\sigma^2 \le c_q c_\Delta^2$, in which case $c_\Delta\mu/\sigma^2 \ge \mu/c_qc_\Delta$ and
 $$\log\kappa \ge (x-c_\Delta)\Gamma^{-1}(\mu/c_qc_\Delta)/2c_\Delta.$$
If we fix $\mu,C_\mu$ and let $c_\Delta\to 0$ while $c_qc_\Delta^{\alpha} \le C$, then $\mu/c_qc_\Delta \ge \mu c_\Delta^{\alpha-1}/C$ and
$$\log \kappa \ge (x-c_\Delta)\Gamma^{-1}(\mu c_\Delta^{\alpha-1}/C)/2c_\Delta = \psi(\mu c_\Delta^{\alpha-1}/C)(x-c_\Delta)\mu c_\Delta^{\alpha-2}/2C.$$
If $\alpha=1$ then $\psi(\mu c_\Delta^{\alpha-1}/C) = \psi(\mu/C)$ is constant, so $\kappa$ grows exponentially in $1/c_\Delta$. If $1<\alpha \le 2$ then $c_\Delta^{\alpha-1} \to 0$, so for $\mu c_\Delta^{\alpha-1}/C \le \log 2$, $\log\kappa \ge (x-c_\Delta)\mu c_\Delta^{\alpha-2}/2C$, and $\kappa$ grows exponentially in $c_\Delta^{\alpha-2}$ if $\alpha<2$. If $\alpha=2$ then since $\lim_{y\to 0}\psi(y)=2$, $\liminf_{\Delta \to 0}\log \kappa \ge x\mu c_\Delta/C$, similar to the case $\mu/\sigma^2$ fixed. Note that if $1 \le \alpha<2$ then $\kappa,\lfloor\kappa\rfloor t_0 \to \infty$, so if $X_0 \le x/2$ then $(\1(X_t \geq x))_{t \ge 0}$ converges weakly to the identically zero process. If $X$ is a Markov chain and $\alpha=1$ then since $X_t'=0$, $|\mu_t(X)| \leq c_qc_\Delta$ so we can take $C_\mu=C$.

\item \textit{Drift escape.} Here, $\gamma = \ep \mu/\sigma^2$ and $a=bx$, where $T = (1+b)x/(1-\ep)\mu$ is the amount of time we allow $X_t$ to remain below $x$. If we let $c_\Delta \to 0$ with $\mu,\sigma^2,\ep,b$ fixed then $\log\kappa \to 2\ep \mu bx/\sigma^2$. If $c_q>0$ bounds the transition rate then
$$\log \kappa \ge bx\Gamma^{-1}(\ep\mu/c_qc_\Delta)/c_\Delta = bx\psi(\ep\mu/c_qc_\Delta)\mu/c_qc_\Delta^2.$$
If we let $c_\Delta \to 0$ while $c_qc_\Delta^{\alpha} \le C$ we obtain similar limits as in the previous case, as we vary $\alpha$. We note that if $\alpha<2$, then taking $\ep,b \to 0^+$ as $c_\Delta\to 0$ slowly enough that $\log \kappa \to \infty$, we find that if $X_0 \ge 0$ then $\lim_{c_\Delta\to 0}\1( T_x > x/\mu)=0$ in probability.

\item \textit{Diffusive barrier.} Here, $\gamma = x/2\langle X \rangle_t$ and $a=x/2$. If we fix $T$ and an upper bound on $\langle X \rangle_t>0$ and let $c_\Delta \to 0$ then $\log \kappa \to x^2/2\langle X \rangle_t$. If $c_q>0$ bounds the transition rate then $\langle X \rangle_t \le c_qc_\Delta^2T$ and
$$\log \kappa \ge \Gamma^{-1}(x/2c_qc_\Delta T)x/2c_\Delta = \psi(x/2c_qc_\Delta T)x^2/4c_qc_\Delta^2T.$$
Letting $c_\Delta\to 0$ with $c_qc_\Delta^{\alpha} \le C$, $\log \kappa \ge \psi(xc_\Delta^{\alpha-1}/2CT)x^2c_\Delta^{\alpha-2}/4CT$ and the scaling behaviour is the same as in the drift barrier case. In particular, for $\alpha<2$, letting $T\to \infty$ slowly enough as $c_\Delta \to 0$ that $\kappa \to \infty$, the process $(\1(\sup_{t \le T}X_t))_{T \ge 0}$ converges weakly to the zero process, when $X_0 \le 0$.
 
\item \textit{Diffusive escape.} Here, $\gamma = (\sigma^2/4)(1/(2x+c_\Delta)^2\rho)$ and $a=bx^2$. If we fix $b,\sigma^2$ and let $c_\Delta\to 0$ then $\gamma=1$ for $c_\Delta$ small enough and so $\log \kappa \to 2bx^2$. In addition, since $\sigma^2_t(X) \le \rho_t(X)$ and $\sigma^2_t(X) \ge \sigma^2$ by assumption, we have the constraint $q_tc_{\Delta,t}^2 = \rho_t\geq \sigma^2$, which is not satisfied when $c_qc_\Delta^{\alpha} \le C$ with $\alpha<2$.
\end{enumerate}

\noindent\textbf{Stochastic logistic model.} We define the Markov chain $X$ on $\{0,\dots,n\}$ with
$$X \to \begin{cases} X+1 \quad\hbox{at rate}& \lambda n^{-1}X(n-X) \\
X-1 \quad\hbox{at rate} & X \end{cases}$$
where $\lambda \in \R_+$ and $n>0$ is an integer parameter, and $\lambda$ is allowed to depend on $n$. We can represent $X_t$ as the number infectious in the following process. There are $n$ individuals, each healthy or infectious. Each infectious individual becomes healthy at rate 1, and infects each healthy individual at rate $\lambda n^{-1}$. Our interest is in the time to extinction
$$\tau = \inf\{t:X_t=0\}.$$
This model has been studied in detail -- see \cite{brightzak} for recent work and a survey of existing research.
 Letting $\delta = 1-\lambda$, the main result of \cite{brightzak} concerns the subcritical regime where $\lim_{n\to\infty}n^{-1/2}\delta =\infty$. They show that, subject to the assumption $\lim_{n\to\infty}\delta x_0 =\infty$,
 $$\delta \tau  - (\log n + 2 \log \delta - \log ( 1 + \delta n / \lambda x_0) - \log \lambda) \to W \quad \hbox{as} \quad n\to\infty,$$
 where $W$ is the standard Gumbel, with distribution $\P( W \le w ) = e^{-e^{-w}}$.
 Letting $\delta_0 = n^{1/2}\delta$ and $x = n^{-1}X$, this estimate is carried out in three phases:
 \begin{enumerate}[noitemsep]
 \item The early phase, when $n^{1/2}x \ge \delta_0^{5/4}$,
 \item The intermediate phase, when $\delta_0^{1/4} \le n^{1/2} x \le \delta_0^{5/4}$, and
 \item The final phase, when $n^{1/2}x \le \delta_0^{1/4}$.
 \end{enumerate}
The early and final phases are simpler to study, and there appears to be only one natural proof in each case, which the authors have given. Since the intermediate phase is more complex, multiple proofs are possible, and we give an alternate, and in our opinion somewhat simpler, proof using the passage time estimates developed above. We will assume, as they do, that $\lambda$ is bounded above $0$.\\

Let $\phi_t$ denote the flow corresponding to the differential equation
$$x' = f(x) \ \hbox{with} \ f(x)= \lambda x(1-x)-x = -x(\delta + \lambda x),$$
so that $\phi_0(x)=x$ and $\partial_t\phi_t(x) = f(\phi_t(x))$ for $t\ge 0$. Since $t\mapsto \phi_t(x_0)$ is decreasing, let $t(x_0,x)$ be the unique value of $t$ so that $\phi_t(x_0)=x$, then let $x^* = \delta_0^{1/4}n^{-1/2}$ and let $t^* = t(x_0,x^*)$. Then, the precise statement of the estimate in the intermediate phase is as follows -- note $o(1)$ is as $n\to\infty$.
\begin{equation}\label{eq:intphase}
\hbox{if} \quad \delta_0^{1/4} \le n^{1/2}x_0 \le \delta_0^{5/4}, \quad \hbox{then} \quad \P(|x_{t^*}-x^*| > \delta_0^{1/6}n^{-1/2}) = o(1).
\end{equation}
To prove this we first define the process $y$ by $y_t = x_t - \phi_t(x_0)$, so that
$$\mu_t(y) = \mu_t(x) - f(\phi_t(x_0)) = -y_t(\delta +\lambda (x_t+\phi_t(x_0))),$$
which follows after factoring the difference of squares. Since $t\mapsto \phi_t(x_0)$ is continuous, $\sigma^2(y) = \sigma^2(x)$, which we easily compute and then bound above as
$$\sigma^2(x) = n^{-1}(\lambda x(1-x)+x) \le n^{-1}(1+\lambda)x.$$
In particular, we find that
$$-\sgn(y)(\mu/\sigma^2)(y) \ge ny \lambda/(1+\lambda),$$
Next, we rescale time by $1/(\delta + \lambda(x_t+\phi_t(x_0)))$ so that $\mu(y)\equiv -y$. In other words, we define a new time variable $s$ given by
$$s(t) = \int_0^t (\delta + \lambda(x_r + \phi_r(x_0)))dr,$$
and then look at $(y_s)$ instead of $(y_t)$. It is easy to check that $(y_s)_{s \ge 0}$ is still an hjp. The rescaling has no effect on the ratio $(\mu/\sigma^2)(y)$, since they scale by the same amount. Thus, after rescaling,
$$\sigma^2(y) \le \sigma^2 := n^{-1}(1+\lambda)/\lambda$$
and if $\ep \le |y| \le  2\ep$ then $-\sgn(y)\mu(y) \ge \ep$ and $|\mu(y)| \le 2\ep$. Using Lemma \ref{lem:driftbar} twice, on $y$ and $-y$, with $s_0 = (\ep-n^{-1})/40 \ep$, $\gamma = n\ep \lambda/(1+\lambda)$, $a=(\ep-n^{-1})/2$, and noting $y_0=0$,
\begin{equation}\label{eq:intphaseest}
\P(\sup_{s < \lfloor \kappa^{1/2} \rfloor s_0} |y_t| \ge 2\ep) \le 6/\kappa^{1/2}
\end{equation}
with $\kappa = \kappa_{n^{-1}}(\gamma,a)$. Taking $\ep = \delta_0^{1/6}n^{-1/2}/2$, since $\ep \gg n^{-1}$, $s_0 \to 1/40$ and $a \sim \ep/2$. Since $\lambda \ge 0$, $\delta = 1-\lambda \le 1$ and $\ep = \delta_0^{-5/6}\delta/2 = o(1)$. In this case, since $c_\Delta = n^{-1}$, $c_\Delta\gamma = o(1)$ and so
\begin{equation}\label{eq:intphasekappa}
\log \kappa \sim 2\gamma a \sim n\ep^2\lambda/(1+\lambda) = \delta_0^{1/3}\lambda/(4(1+\lambda)).
\end{equation}
Since, by assumption, $\lambda$ is bounded above zero, $\log \kappa/\delta_0^{1/3}$ is bounded above zero. In particular, $\log\kappa \to \infty$ as $n\to\infty$ and the right-hand side of \eqref{eq:intphaseest} is $o(1)$. Thus, to establish \eqref{eq:intphase} it remains to check that
$$\sup_{\delta_0^{1/4} \le n^{1/2}x_0 \le \delta_0^{5/4}}s(t^*) \le s_0 \lfloor \kappa^{1/2} \rfloor \quad \hbox{for} \ \ n \ \ \hbox{large enough}.$$
In the region of interest, $\delta_0^{1/4} \le n^{1/2}\phi_t(x_0) \le \delta_0^{5/4}$, and on the event of interest,
$$\sup_{t \le t^*}|x_t - \phi_t(x_0)| \le 2\ep = \delta_0^{1/6}n^{-1/2} \le \delta_0^{-1/12}\phi_t(x_0) = o(\phi_t(x_0)),$$
which, noting $\delta = \delta_0n^{-1/2}$, gives
$$s(t^*) \le (\delta_0 + (2+o(1))\lambda \delta_0^{5/4})n^{-1/2}t^* = O(\delta_0^{5/4}n^{-1/2}t^*) = O(\delta_0^{1/4}\delta t^*).$$
Solving the equation $x'=f(x)=-x(\delta + \lambda x)$ by separation of variables, we find that $\phi_t$ satisfies
 $$t = \delta^{-1}\left(\log \frac{x_0}{\phi_t(x_0)} - \log\frac{\delta + \lambda x_0}{\delta + \lambda \phi_t(x_0)} \right),$$
at the second term is $\le 0$ since $x_0 \ge \phi_t(x_0)$ for $t\ge 0$. Thus for $x_0 \le \delta_0^{5/4}n^{-1/2}$, $x_0/\phi_{t^*}(x_0) \le \delta_0$ and
$$\delta t^* \le \log \delta_0, \ \hbox{which implies} \ s(t^*) = O(\delta_0^{1/4} \log \delta_0) = o(s_0\lfloor\kappa^{1/2}\rfloor).$$

\section{Proofs}\label{sec:proof}

We begin with a useful fact regarding the filtration. By definition of $\F$, for $t<r$, on the event $E(t,r) = \{J(t,r)=\varnothing\}$, $\F_r$ coincides with $\F_t$. That is,
$$\{B \cap E(t,r): B \in \F_r\} = \{B \cap E(t,r):B \in \F_t\}.$$
In particular,
\begin{equation}\label{eq:stretch-meas}
\hbox{on} \ E(t,r) \ \hbox{and for} \ v \in [t,r), \ v\mapsto (q_v,D_v,\Delta_v) \ \hbox{is} \ \F_t-\hbox{measurable}.
\end{equation}
In words, what this means is that, given information up to the most recent jump, the data are deterministic until the next jump occurs. For $t \ge 0$ let $J_1(t) = \inf J(t,\infty)$ denote the first jump time after $t$.
 It follows that, $J_1(t)$ has the $\F_t$-measurable density function $q_r e^{-\int_t^r q_v dv}, \ r>t.$
 Left-continuity of $t\mapsto \Delta_t(X)$ then implies that $\Delta_{J_1(t)}(X)$ is determined by $\F_t$ and $J_1(t)$.\\
 
\begin{proof}[Proof of Theorem \ref{thm:compensate}]
Define the stopped processes $X^{\tau_n}$ given by $X_t^{\tau_n} = X_{t \wedge \tau_n}$, where
 $$\tau_n = \inf\{t:q_t \ge n \ \hbox{or} \ \int_0^t |D_s| = n\}$$
is a localizing sequence for $\zeta$. Note that $X^{\tau_n}$ is defined for $t \ge 0$ and is an hjp with data
$$(q_t,D_t,\Delta_t)(X^{\tau_n}) = (q_t,D_t,\Delta_t)(X)\cdot\1(\tau_n > t).$$ 
Moreover,
\begin{equation}\label{eq:comp1}
 \sup_{t \ge 0} q_t(X^{\tau_n}) \le n \quad \hbox{and} \quad \int_0^{\infty}|D_t(X^{\tau_n})|dt \le n.
 \end{equation}
To obtain the theorem it is enough to show that for $s<t$ and each $n$, both sides of the equation
$$\E[X_t^{\tau_n} - X_s^{\tau_n} \mid \F_s] = \E\left[ \int_s^t \mu_r(X^{\tau_n})dr \ \mid \ \F_s \right]$$
exist, and equality holds. For ease of notation, fix $n$ and let $X$ denote $X^{\tau_n}$, with data $q,D,\Delta$. The left- and right-hand sides are given respectively by
$$\begin{array}{rcl}
X(s,t) &=& D(s,t) + \Delta(s,t) \ \hbox{and} \\
\mu(s,t) &=& D(s,t) + \bar\Delta(s,t)
\end{array}$$
where
$$\begin{array}{rcl}
D(s,t) &=& \E[ \ \int_s^t D_rdr \ \mid \ \F_s \ ], \\
\Delta(s,t) &=& \E[ \ \sum_{(r,u) \in K(s,t]}\Delta_r(u) \ \mid \ \F_s \ ]\ \hbox{and}\\
\bar\Delta(s,t) &=& \E[ \ \int_s^t q_r\int_0^1\Delta_r(u)dudr \ \mid \ \F_s \ ].
\end{array}$$
Using \eqref{eq:comp1} and $\|\Delta_t\|_{L^{\infty}} \le c_\Delta$,
$$\begin{array}{rcl}
\int_s^t |D_r|dr &\le & n,\\
\sum_{(r,u) \in K(s,t]}\Delta_r(u) &\preceq & c_\Delta\poi(n(t-s)) \ \hbox{and}\\
\int_s^t q_r\int_0^1\Delta_r(u)dudr \ &\le & c_\Delta n(t-s),
\end{array}$$
where $\preceq$ denotes stochastic domination, so all three functions above are integrable. To complete the proof it remains to show $\Delta(s,t) = \bar\Delta(s,t)$. From the bounds on $q_t,\Delta_t$,
\begin{equation}\label{eq:jumps1}
\begin{array}{rcl}
\Delta(s,t+h)-\Delta(s,t) &=& \E[\sum_{K(t,t+h]}\Delta_r(u) \mid \F_s] \\
&=& \E[\sum_{K(t,t+h]}\Delta_r(u) \mid \F_s, \ |J(t,t+h]|=1] + O(h^2)
\end{array}
\end{equation}
where $O(h^2) \le c_\Delta \E[N;N\ge 2]$ with $N \stackrel{d}{=} \poi(nh)$, and in particular is uniform in $s,t$. Next we use \eqref{eq:stretch-meas} and the ensuing facts about $J_1(t)$ and $\Delta_{J_1(t)}(X)$ to compute
\begin{equation}\label{eq:jumps2}
\begin{array}{rcl}
\E[\sum_{K(t,t+h]}\Delta_r(u) \mid \F_t, \ |J(t,t+h]|=1] 
&=& \int_t^{t+h} q_re^{-\int_s^r q_v dv} \left(\int_0^1\Delta_r(u)du\right)dr\\
&=& \int_t^{t+h} q_r \int_0^1 \Delta_r(u)du dr + O(h^2)\\
\end{array}
\end{equation}
where $O(h^2) \le c_\Delta (nh)^2$ and so uniform in $t$. Taking $\E[\cdot \mid \F_s]$ in \eqref{eq:jumps2}, then combining with \eqref{eq:jumps1}, we find
$$\Delta(s,t+h)-\Delta(s,t)= \bar\Delta(s,t+h)-\bar\Delta(s,t)+ O(h^2)$$
with $O(h^2)$ uniform in $t$. Using a standard trick, we fix $t$ and integer $N>0$ and let $h=(t-s)/N$, then make a telescoping sum and note $\Delta(s,s)=\bar\Delta(s,s)=0$ to find
$$\Delta(s,t)-\bar\Delta(s,t) = O(Nh^2) = O((t-s)^2/N).$$
Since $t-s$ is fixed, letting $N\to\infty$ gives $\Delta(s,t)=\bar\Delta(s,t)$ as desired.
\end{proof}

\begin{proof}[Proof of Theorem \ref{thm:hmp}]
We first note the characterization of $\zeta$ in case $q:\X\to\R_+$ is locally bounded, which is straightforward. Since $D$ is locally Lipschitz by assumption, it is in particular locally bounded since
$$\sup_{x \in \X:|x| \le r}|D(x)| \le |D(o)| + L_r|x|,$$
where the origin $o$ is the unique element in $\X$ with $|o|=0$, and $L_r$ is a Lipschitz constant for $D$ on the ball $B_r$. Thus if $X$ remains bounded up to some time $\eta$, so do $q(X)$ and $|D(X)|$, which also makes $\int_0^t |D_s(X)|ds$ locally finite, and implies $\eta\le\zeta$. On the other hand, if $|X_t| \to \infty$ as $t\to\eta^-$ then either $q_t(X)$ remains bounded, in which case $|\sum_{(s,u) \in K_t}\Delta(X_{s^-}(u)| \preceq c_\Delta \poi(nt)$ which is a.s. finite, implying that $\int_0^t D_s(X)ds$ must have diverged, or else $q_t(X)$ becomes unbounded, which together imply $\zeta \le \eta$ and completes the characterization of $\zeta$ in this case.\\

Next, we want to find a unique $X$ satisfying the equation
\begin{equation}\label{eq:jumpDE}
X_t = x_0 + \int_0^tD(X_s)ds + \sum_{(s,u) \in K_t}\Delta(X_{s^-})(u)
\end{equation}
for $t<\zeta$. This then implies that $\mu_t(X) = \mu(X_t)$, where $\mu:\X\to \X$ is defined by
$$\mu(x) = D(x) + \int_0^1 \Delta(x,u)du.$$
We now show existence and uniqueness in \eqref{eq:jumpDE}. Let $t_1,t_2,\dots$ denote the jump times of $(K_t)$ and $u_1,u_2,\dots$ the corresponding $\unif[0,1]$ random variables. Suppose $t_1,\dots,t_i$ are known, $\{q_t,X_t:t<t_i\}$ are known and unique given $x_0$ and that $X_{t_i^-} = \lim_{t \to t_i^-}X_t$
 exists. Equation \eqref{eq:jumpDE} then gives $X_{t_i} = X_{t_i^-} + \Delta(X_{t_i^-},u_i)$. For $t \in [t_i,t_{i+1})$, $(X_t)$ satisfies the deterministic integral equation
$$X_t = X_{t_i} + \int_0^tD(X_s)ds.$$
which, since $D$ is locally Lipschitz, has a unique continuous solution on the interval $[t_i,\eta)$ where $\eta = \lim_{n\to\infty}\inf\{t>t_i:|X_t|\ge n\}$ is the escape time of the solution. This determines $X_t$ and $q_t=q(X_t)$ for $t \in [t_i,t_{i+1} \wedge \eta)$. If $t_{i+1}<\eta$ then by continuity of solutions $X_{t_{i+1}^-}$ exists and we repeat the induction step. Otherwise, $\zeta=\eta$ and we are done.
\end{proof}

\begin{proof}[Proof of Theorem \ref{thm:stocalc}]
It is a straightforward exercise to check the given processes satisfy the conditions of an hjp, with the the data as shown, so we omit the proof. It remains to check the Taylor approximation. Use Taylor's theorem to find
$$\Delta_t(f(X))(u) = f'(X_t)\Delta_t(X)(u) + \frac{1}{2}f''(X_t^*(u))\Delta_t^2(X)(u),$$
for some $X_t^*(u)$ with $|X_t^*(u)-X_t| \leq \Delta_t(X)(u)$. In particular,
$$\left|\Delta_t(f(X))(u) - f'(X_t)\Delta_t(X)(u) \right | \le \frac{1}{2}\Delta_t^2(X)(u)\sup_{|x-X_t| \le c_\Delta}|f''(x)|.$$
Integrating and using the triangle inequality,
$$\left |\int_0^1\Delta_t(f(X))(u)du - f'(X_t)\int_0^1\Delta_t(X)(u)du \right | \le \frac{1}{2}\sigma_t^2(X)\sup_{|x-X_t| \le c_\Delta}|f''(x)|.$$
The result then follows from the formula for $\mu_t(f(X))$. 
\end{proof}

Before tackling the proof of Theorem \ref{thm:sampath} we show how to compensate an hjp in various ways to obtain martingales. Recall the \emph{compensated process} $M(X)$ and the \emph{predictable quadratic variation} $\langle X \rangle$ of an hjp $X$, given by
$$M_t(X) = X_t - X_0 - \int_0^t\mu_s(X)ds \quad \hbox{and} \quad
\langle X \rangle_t = \langle X,X \rangle_t = \int_0^t\sigma^2_s(X)ds.$$

\begin{lemma}\label{lem:mart}
Let $X$ be an hjp. Then the following processes $M(X),Q(X),E(X,\lambda)$, with $\lambda \in \R$ fixed, are hjp and are local martingales.
$$\begin{array}{rrcl}
\hbox{Compensated process.}& M_t(X) &=& X_t - X_0 - \int_0^t\mu_s(X)ds \\
\hbox{Compensated quadratic.}& Q_t(X) &=& M_t(X)^2 - \langle X \rangle \\
\hbox{Compensated exponential.} & E_t(X,\lambda) &=& \exp(\lambda (X_t - X_0) - \int_0^t e^{-\lambda X_s}\mu_s(e^{\lambda X})ds)
\end{array}$$
\end{lemma}

\begin{proof}
It is easy to check as we go along that the processes are hjp, using Theorem \ref{thm:stocalc}. The process $M(X)$ is the compensated process from Theorem \ref{thm:compensate}. For $Q(X)$, we first verify from the product rule and the fact $\sigma^2(X) = \sigma(X,X)$ that
$$\mu(X^2) = 2X \mu(X) + \sigma^2(X).$$
Since $t\mapsto \int_0^t\mu_s(X)ds$ is continuous, $\Delta(M(X)) = \Delta(X)$, so $\sigma^2(M(X)) = \sigma^2(X)$ and $\langle M(X) \rangle = \langle X \rangle$. Since $\mu(M(X))=0$ we find
$$\mu(M(X)^2) = \sigma^2(M(X)) = \sigma^2(X).$$
Since $M_0(X) = \langle X \rangle_0 = 0$, $Q_0(X) =0$. Since $\mu(\langle X \rangle) = \sigma^2(X)$,
$$\mu(Q(X)) = \mu(M(X)^2) - \mu(\langle X \rangle) = \sigma^2(X) - \sigma^2(X) =0,$$
so $Q(X)$ is a local martingale. For $E(X,\lambda)$ we first consider, more generally, a positive hjp $Y$. 
 We want to find a positive hjp $Z$ with $\Delta(Z) \equiv 0$ so that $\mu(Y/Z) \equiv 0$. Since $Z$ is assumed to have no jumps, $\sigma(Y,1/Z) \equiv 0$, so the product rule holds in its usual form and gives
$$\mu(Y/Z) = \mu(Y)/Z - Y\mu(Z)/Z^2.$$
Setting $\mu(Y/Z) \equiv 0$ and solving for $\mu(Z)$ gives
$$\mu(Z)/Z = \mu(Y)/Y.$$
Since $Z$ has no jumps, $t\mapsto\mu_t(Z)$ is the a.e. derivative of $t\mapsto Z_t$, so the chain rule and the fundamental theorem of calculus hold in their usual form, and we find $\mu(\log Z) = \mu(Z)/Z$ and
$$Z_t = Z_0\cdot\exp(\int_0^t \mu_s(\log Z)ds) = Z_0\cdot \exp(\int_0^t \mu_s(Z)/Z_s ds).$$
Taking the initial data $Z_0=Y_0$ and using the relation $\mu(Z)/Z = \mu(Y)/Y$, we obtain
$$Z_t = Y_0\cdot\exp \left( \int_0^t \mu_s(Y)/Y_s ds \right).$$
Then, to get $E(X,\lambda)$, take $Y/Z$ with $Y$ given by $Y_t = e^{\lambda X_t}$ in the above.
\end{proof}

\begin{proof}[Proof of Theorem \ref{thm:sampath}]
Since $E(X,\lambda)$ is a local martingale, let $(\tau_n)_{n \ge 0}$ be a localizing sequence for $\zeta$, such that $E(X,\lambda)^{\tau_n}$ is a martingale for each $n\ge 0$. Using Doob's inequality and the fact $E_t(X,\lambda) = E_t(X,\lambda)^{\tau_n}$ for $n\ge 0$ and $t \le \tau_n$, we find that for $n\ge 0$, $\lambda \in \R$ and $a>0$,
\begin{equation}\label{eq:sam1}
\P(\sup_{t\le \tau_n} E_t(X,\lambda) \ge e^{|\lambda| a}) \le e^{-|\lambda| a}.
\end{equation}
Since $n\mapsto \tau_n$ is increasing, the above event is increasing in $n$, and since $\tau_n \to \zeta$, using monotone convergence of probability, it follows that the same estimate holds with $t<\zeta$ in place of $t \le \tau_n$. Using the Taylor approximation of Theorem \ref{thm:stocalc},
$$|\mu_s(e^{\lambda X}) - \lambda e^{\lambda X_s}\mu_s(X)|\le \frac{1}{2}\lambda^2 e^{\lambda X_s + |\lambda|c_\Delta}\sigma^2_s(X),$$
which means
$$|e^{-\lambda X_s}\mu_s(e^{\lambda X}) - \lambda\mu_s(X)| \le \frac{1}{2}\lambda^2e^{|\lambda|c_\Delta}\sigma_s^2(X).$$
Recalling the definition of $E(X,\lambda)$ and recognizing $M(X)$ and $\langle X\rangle$,
$$\begin{array}{rcl}
E_t(X,\lambda) &=& \exp(\lambda (X_t - X_0) - \int_0^t e^{-\lambda X_s}\mu_s(e^{\lambda X})ds) \\
&\geq & \exp(\lambda M_t(X) - \frac{1}{2}\lambda^2e^{|\lambda| c_\Delta}\langle X \rangle_t).
\end{array}$$
Dividing through by $|\lambda|$ in the event from \eqref{eq:sam1} and considering separately $\lambda>0$ and $\lambda<0$ gives the desired result.
\end{proof}

\bibliographystyle{plain}
\bibliography{stocalc}

\end{document}